\documentclass[eqno,a4paper,11pt]{article}
\title{Singularity formation for the 1D compressible Euler equation with  variable damping coefficient}
 \author{Yuusuke Sugiyama\footnote{e-mail:sugiyama@ma.kagu.tus.ac.jp\ \ \ telephone number:(+81)3-3260-4271}\\
Department of Mathematics,\ Tokyo University of Science\\
Kagurazaka 1-3, Shinjuku-ku, Tokyo 162-8601, Japan}
\date{}
\usepackage{amssymb,amsmath,amsthm,mathrsfs,delarray,enumerate}
\usepackage{graphics}

\theoremstyle{definition} 
\newtheorem{Def}{Deffinition}[section]

\newtheorem{Lemma}[Def]{Lemma}
\newtheorem{Thm}[Def]{Theorem}
\newtheorem{Remark}[Def]{Remark}

\newcommand{\R}{\mathbb{R}} 

\makeatletter

\@addtoreset{equation}{section}

\begin{document}
\maketitle 
\begin{abstract}
In this paper, we consider some blow-up problems for the 1D Euler equation with time and space dependent  damping.
We investigate sufficient conditions on  initial data and the rate of spatial or
time-like decay of the coefficient of damping for the occurrence of the finite time blow-up. In particular, our sufficient conditions ensure that  the derivative blow-up occurs in finite time
with the solution itself and the pressure bounded.  Our method is based on  simple estimates with Riemann invariants. Furthermore, we  give sharp lower and upper estimates of the lifespan
of solutions, when initial data are small perturbations of constant states.

\end{abstract}

\section{Introduction}
In this paper, we consider the following Cauchy problem of the compressible Euler equation with time and space dependent  damping
\begin{eqnarray} 
\left\{  \begin{array}{ll} \label{de0}
  u_t - v_x  =0,\\
 v_t + p(u)_x = -a(t,x)v , \\   
   (u(0,x), v(0,x))=(u_0 (x), v_0(x)).
\end{array} \right.  
\end{eqnarray} 
Here $x \in \R$ is the Lagrangian spatial variable and $t \in \R_+$ is time. $u=u(t,x)$ and $v=v(t,x)$  are the real valued unknown functions, which stand for the specific volume and the fluid  velocity.
For the damping coefficient $a(t,x)$, we suppose that $a(t,x)\geq 0$ for all $(t,x)$. 
In the case with $a\equiv 0$, the equations in \eqref{de0} are the compressible Euler, which is a fundamental model for the compressible inviscid fluids.
In the case with $a\not\equiv 0$, this system describes the flow of fluids in  porous media.
We assume that the flow is barotropic ideal gases. Namely the pressure $p$ satisfies that 
\begin{eqnarray}\label{poly}
p(u)=\frac{u^{-\gamma}}{\gamma} \ \  \mbox{for} \  \gamma > 1. 
\end{eqnarray}
For initial data, in order to avoid the singularity of $p'$, we assume that there exists  constant $\delta_0 >0$  such that for all $x \in \R$
\begin{eqnarray} \label{nosin}
u_0 (x) \geq \delta_0.
\end{eqnarray}
The local existence and the uniqueness of solutions of \eqref{de0} hold with $C^1 _b$ initial data (e.g. Friedrichs \cite{KF}, Lax \cite{lax0} and Li and Wu \cite{LW}).
If $u_0, v_0 \in C^1 _b (\R)$ and $a \in C^1 _b (\R^2)$ and \eqref{nosin} is assumed, then \eqref{de0} has a local and unique solution until 
the one of the following three blow-up phenomena.
The first is the $L^\infty$ blow-up:
\begin{eqnarray*} 
\varlimsup_{t\nearrow T^{*}}  \| (u,v)(t)\|_{L^{\infty}}=\infty.
\end{eqnarray*} 
The second is the derivative blow-up:
\begin{eqnarray*} 
\varlimsup_{t\nearrow T^{*}} \left( \| (u_t,v_t)(t)\|_{L^{\infty}}+\| (u_x,v_x)(t)\|_{L^{\infty}} \right)=\infty.
\end{eqnarray*} 
The second is the blow-up of $p'$:
\begin{eqnarray*} 
\varlimsup_{t\nearrow T^{*}}  \| p'(u(t))\|_{L^{\infty}}=\infty.
\end{eqnarray*} 
We consider sufficient conditions for the occurrence of the derivative blow-up on $a(t,x)$ and initial data.
The purposes of this paper are as follows.

{\bf (I)}\  For the case of space independent damping: $a(t,x)=\lambda/(1+t)^\mu$, we show that the derivative blow-up  occurs in finite time with $u, v$ and $p'$ bounded
under some suitable condition on initial data, when  $\mu >1$ and $\lambda \geq 0$ or $\mu =1$ and $0 \leq  \lambda  \leq 2$ (Theorems \ref{main2} and \ref{main1}).

{\bf (I\hspace{-.1em}I)}\ For the same coefficient of damping as in (I), we give sharp upper and lower estimates of the lifespan of solutions, when initial data are small perturbations near constant states (Theorems \ref{main3} and \ref{main4}).

{\bf (I\hspace{-.1em}I\hspace{-.1em}I)}\  For the case of time and space dependent damping, we show that the derivative blow-up  occurs in finite time with $u, v$ and $p'$ bounded
under some suitable condition on  initial data, when the time-like or spatial decay-rate of $a(t,x)$, $a_t (t,x)$ and $a_x (t,x)$  are faster than $-1$ (Theorem \ref{main5}).

Before recalling proceeding results, we give notations. We set $c=\sqrt{-p'(u)}$ and $\eta= \int_{u} ^{\infty} c(\xi) d\xi=\frac{2}{\gamma-1} u^{-(\gamma-1)/2}$
and define Riemann invariants as follows:
\begin{eqnarray} \label{ri}
\begin{array}{ll}
r = v-\eta, \\
s = v+\eta.
\end{array}
\end{eqnarray}

We define the lifespan of solutions by
\begin{align*}
T^* =& \sup\{T>0 \ | \   \sup_{t \in [0,T )} \| (u,v)(t)\|_{L^{\infty}} +  \| (u_t,v_t)(t)\|_{L^{\infty}} \\
&+\| (u_x,v_x)(t)\|_{L^{\infty}} +  \| p'(u(t))\|_{L^{\infty}} < \infty \}.
\end{align*}

First, we recall known results for \eqref{de0} without damping ($a(t,x)\equiv 0$).
For more general $2\times 2$ strictly hyperbolic system including the 1D Euler equation, sufficient conditions for the derivative blow-up (the formation singularity) has been studied by many mathematicians
(e.g. Lax \cite{lax}, Zabusky \cite{z}, Klainerman and Majda \cite{km}, MacCamy and Mizel \cite{mm2}, Manfrin \cite{mm}, Liu \cite{liu} and 
 Cheng,  Pan and  Zhu \cite{GPZ}).  In  \cite{lax}, Lax has established important formulas for solutions to $2 \times 2$ hyperbolic systems. Applying the Lax formulas to the Euler equation,  we can show that  if $r_x$ or $s_x$ is negative at some point, then the derivative blow-up occurs. Furthermore, since the Riemann invariants do  not change on the plus and minus, we can get the boundedness of $u, v$ and $p'$ under suitable assumptions on $r(0,x)$ and $s(0,x)$ (see also Manfrin \cite{mm}).
Lax's method also can show that solutions exist globally in time, if $r_x (0,x)$ and $s_x (0,x)$ are non-negative for all $x$.
 Recently, in \cite{GPZ}, Cheng,  Pan and  Zhu have shown that smooth solutions to the 1D Euler equation with $1<\gamma <3$ exist globally in time, if and only if  $r_x$ and $s_x$ are non-negative.  Moreover they have shown that the derivative blow-up occurs with $u$, $v$ and $p'$ bounded, if $r_x$ or $s_x$ is negative at some point.
 Their proof is based on a new time-dependent estimate of $u$.

Next, let us recall  known results for the 1D Euler equation with constant damping ($a(t,x)\equiv \mbox{potitive constant}$).
Hsiao and  Liu \cite{HL} has proved that  solutions exist globally in time,  if initial data are small perturbations near constant states and that  small solutions asymptotically behave like
that to the following porous media system as $t\rightarrow \infty$:
\begin{eqnarray*} 
\left\{  \begin{array}{ll} 
 \bar{u}_t = -p(\bar{u})_{xx},\\
 \bar{v}=-p(\bar{u})_x .
\end{array} \right.  
\end{eqnarray*} 
After this work, many improvements and generalizations of this work have been investigated (see Hsiao and  Liu \cite{HL2}, Nishihara \cite{NK}, Hsiao and Serre \cite{HS} and
Marcati  and Nishihara \cite{MN}).
We note that, in the above papers  for the 1D Euler equation with constant damping, they assume that the existence of the limit $\lim_{x \rightarrow \pm \infty} (u_0 (x), v_0 (x))$ and that the convergence rate of the limit
is sufficiently fast in order to show the global existence of solutions via $L^2$ energy estimates.
From Wang and Chen \cite{WC}, it is known that the derivative blow-up can occur generally for the compressible Euler-Poisson equation with damping (including the 1D Euler equation with constant damping).

Next we recall the 1D Euler equation with space independent damping:
\begin{eqnarray} 
\left\{  \begin{array}{ll} \label{det}
  u_t - v_x  =0,\\
 v_t + p(u)_x = -\dfrac{\lambda}{(1+t)^\mu} v , \\    
\end{array} \right.  
\end{eqnarray} 
where $\mu, \lambda \geq 0$.
In  \cite{XP1, XP2, XP3}, Pan has found thresholds of $\mu$ and $\lambda$ separating the existence and the nonexistence of global solution in small data regime.
Namely, in the case with $0\leq \mu <1$ and $\lambda >0$ or $\mu=1$ and $\lambda >2$,  Pan \cite{XP2} has proved  that solutions  of \eqref{det} exist globally in time,  if initial data are small and compact perturbations of  constant states. While, in the case with $\mu >1$ and $\lambda >0$ or $\mu=1$ and $0 \leq \lambda \leq 2$, Pan  has proved that solutions of \eqref{det} can blow up under some conditions on initial data in \cite{XP1, XP3}. However, in \cite{XP1, XP3}, it is not determined which types of blow-up occur and which types of blow-up do not occur. In these papers, the estimate of the lifespan is $T^* \leq \exp(C/\varepsilon^2 )$.
The proofs in \cite{XP1, XP3} are based on Sideris \cite{TS}. For the existence and the nonexistence of global solutions to the 3D Euler equation with space independent damping (the 3D version of \eqref{req} with $a(t,x)=\lambda/(1+t)^{\mu}$), 
we refer to Hou, Witt and Yin \cite{HWY} and Hou and Yin \cite{HY}.

As far as the author knows, there are few result on the Euler equation with time and space dependent damping.
In Cheng \cite{G}, Cheng,  Pan and  Zhu \cite{GPZ}, Zheng \cite{hz} and Pan and Zhou \cite{PZ}, they consider the following non-isentropic compressible Euler equations:
 \begin{eqnarray} 
\left\{  \begin{array}{ll} \label{degp}
  u_t - v_x  =0,\\
 v_t + p(S,u)_x = 0 , \\   
S_t =0,
\end{array} \right.  
\end{eqnarray} 
 where $p(S,u)=e^{S(x)} p^{-\gamma} /\gamma$ with $1<\gamma <3$ and $S(t,x)=S(x)$ is a given function. They assume that $\int_{\R} |S'(x)| dx < \infty$.
Their proofs essentially use the fact that  $S$ is independent of $t$ and  the restriction  $1<\gamma <3$  to show the boundedness of the Riemann invariant.

We explain roughly main theorems for \eqref{de} (Theorems \ref{main2}, \ref{main1}, \ref{main3} \ref{main4} and \ref{main5}) and what improvements from the above known results are.
Theorems \ref{main2}, \ref{main1}, \ref{main3} and \ref{main4} treat solutions of \eqref{det}. Theorems \ref{main2} and \ref{main1} give sufficient conditions for  the boundedness $u$, $v$ and $p'$ and the occurrence of the derivative blow-up.
Theorems \ref{main2} and \ref{main1} are analogies of  the results of Lax  \cite{lax} and Cheng  Pan and  Zhu \cite{GPZ} respectively.
Theorems \ref{main3} gives upper estimates of the lifespan and improves the estimates of lifespan in Pan \cite{XP1, XP3}.
Theorems \ref{main4}, which includes the global existence of solutions,   gives lower estimates of the lifespan for all $\mu \geq 0$ and $\lambda \geq 0$. 
In our main theorems, we do not assume that the rate of convergence of the limit $\lim_{x \rightarrow - \infty} (u_0 (x), v_0 (x))$ is enough fast, which is essentially assumed in Nishihara \cite{NK},  Hsiao and  Liu \cite{HL}, \cite{MN} and Pan \cite{XP2}.
 Theorem \ref{main5} gives a sufficient condition for occurrence the derivative blow-up for solutions
 to \eqref{de0}, when the time-like or spatial decays of $a(t,x)$ itself and its $t$ and $x$ derivatives are faster than $-1$.

 In the proofs of Theorems \ref{main2} and  \ref{main1}, in order to get the boundedness of $v$ and $p'$, 
 we estimate the Riemann invariant. The idea for the estimate is the use of $\Phi (t,x,y) =r (t,x) +s(t,y)$
 and $\Psi (t,x,y) =r (t,x) - s(t,y)$. For these functions, we get a priori estimates:
 $\| \Phi (t) \|_{L^{\infty}} \leq \| \Phi (0) \|_{L^{\infty}}$ and  $\| \Psi (t) \|_{L^{\infty}} \leq \| \Psi (0) \|_{L^{\infty}}$.
In the proof of Theorem \ref{main1} , we show that   $\lim_{x \rightarrow - \infty} u(t,x) = \lim_{x \rightarrow - \infty} u_0 (x)$ by using the Riemann invariant, from which,  we have the boundedness of $u$by above estimate on $\Phi$.  Using Lax type formulas, we derive  differential inequalities for $r_x$ and $s_x$ on characteristic curves,  which yield
the blow-up of $r_x$ or $s_x$. In the proof of Theorem \ref{main1},  an upper estimate of $u$ is shown by the method of Cheng  Pan and  Zhu \cite{GPZ}.
In the same way as the proof of Theorem \ref{main2}, Theorems \ref{main3} and \ref{main4} can be shown.
In the proof of Theorems \ref{main5}, if the coefficient of damping $a(t,x)$ decays time-like, 
using $\Phi (t,x,y)$, we can apply the similar idea to in the proof of Theorem \ref{main1}. While, in the case that 
 $a(t,x)$ decays spatially, by applying a bootstrap argument, we get the boundedness of $\Phi (t,x,y)$ and lower and upper bounds of $u$.

\subsection*{Plan of this paper and notations}

This paper is organized as follows: In section 2, we treat the Cauchy problem for \eqref{det}. We introduce four  theorems (Theorems \ref{main2},  \ref{main1}, \ref{main3} and \ref{main4}) for the achievement of our goals {\bf (I)} and {\bf (I\hspace{-.1em}I)}  and some useful identities and Lemmas for their proofs.  
In section 3, we consider  the Cauchy problem \eqref{de0} and prove  a blow-up theorem  with a suitable decay condition on $a(t,x)$.

For $\Omega \subset \R^n$, $C^1 _b (\Omega)$ are the set of  bounded and continuous functions whose first partial derivatives are also bounded on $\Omega$.
The norm of $C^1 _b (\Omega)$  is $\|f \|_{C^1 _b (\Omega)} = \| f \|_{L^\infty  (\Omega)} + \| (\partial_{x_1} f,\ldots,\partial_{x_n}f) \|_{L^\infty (\Omega)}$.
When $\Omega = \R$, for abbreviation, we denote $\|\cdot  \|_{L^\infty (\Omega)} $ and $\|\cdot  \|_{C^1 _b (\Omega)} $  by $ \|\cdot  \|_{L^\infty}$ and $\|\cdot  \|_{C^1 _b }$
respectively.

 \begin{Remark}{\bf The linear wave equation with space independent damping}.
In \cite{JW1}, Wirth studies the behavior of solutions to 
$$u_{tt} - \Delta u=- \dfrac{\lambda u_t}{(1+t)^{\mu}}.$$ 
In \cite{JW2, JW3}, he give  a threshold of $\mu$  separating  solutions to this equation asymptotically behave like
that to the corresponding hear or wave equation.
In \cite{JW1}, the critical case of the threshold  is studied.
\end{Remark}

 \begin{Remark}{\bf (The Euler equation with  Eulerian coordinate)}.
Changing the coordinate from Lagrangian to Eulerian and putting $\rho=1/u$, we see that  the equations in \eqref{de0} are equivalent to
 \begin{eqnarray} 
\left\{  \begin{array}{ll} \label{req}
  \partial_t \rho - \partial_x ( \rho v) =0,\\
  \partial_t (\rho v) +  \partial_x (\rho v ^2 + p(\rho)) = -a(t,x) \rho v. \\   
 
\end{array} \right.  
\end{eqnarray} 
In Pan \cite{XP1}, \cite{XP2} and \cite{XP3}, \eqref{req} is treated.
\end{Remark}

 \begin{Remark}{\bf (Remarks on the $L^\infty$ blow-up and the blow-up of $p'$ for the p-system)}. \label{bpp}
 For the p-system without damping (the equation in \eqref{de0} with $a\equiv 0$), it is known that
 the blow-up of $v$ and $p'$ does not occur with $C^1 _b$ initial data  for $\gamma \geq 1$, which can be easily shown by 
 the fact that Riemann invariants do not change on characteristic curves.
In Cheng  Pan and  Zhu \cite{GPZ}, they have shown that the blow-up of $u$, which physically means the vacuum in fluids, does not occur for $1<\gamma <3$.
For $\gamma \geq 3$, it is open whether the blow-up of $u$ occurs or not.
While, if $\gamma <1$, the blow-up of $v$ and $u$ does not occur.
When $\gamma <-1$, $p'$ does not diverge at $u=0$ ($u=0$ is a zero point of $p'$).
However, when $p'$ goes to zero, the equations loss the strictly hyperbolicity.
The author's papers \cite{s3, s4} give a sufficient condition that the lack of the strictly hyperbolicity
occurs in finite time.
The p-system with $\gamma <1$  would be meaningful in the study of elastic-plastic materials (e.g. Cristescu \cite{NC} and  Ames and  Lohner \cite{al}). 
\end{Remark}

\section{The Euler equation with space independent damping}
In this section, we consider the following Cauchy problem:
\begin{eqnarray} 
\left\{  \begin{array}{ll} \label{de}
  u_t - v_x  =0,\\
 v_t + p(u)_x = -\dfrac{\lambda}{(1+t)^\mu} v , \\   
  (u(0,x), u(0,x))=(u_0 (x), v_0 (x)).  
\end{array} \right.  
\end{eqnarray} 
Before introducing main theorems of this section, we define a function space $X_{m,\delta}$ with constants $m>0$ and $\delta>0$ by
\begin{align*}
X_{m,\delta} = & \{(u,v) \in C^1 _b (\R) \times C^1 _b (\R) \ | \  \| u \|_{L^{\infty}} + \| v \|_{L^{\infty}} \leq M \\
& \mbox{and} \ u(t,x) \geq \delta \ \mbox{for all} \ x \in \R \}
\end{align*}
and put  $c_0=\sqrt{-p'(u_0)}$ and
\begin{eqnarray*} 
\theta_{\gamma} (u) =\left\{ \begin{array}{ll}
\frac{4}{3-\gamma}  u^{\frac{3-\gamma}{4}}\ \ \mbox{for} \ \gamma \not= 3,\\
\log u \ \ \mbox{for} \  \gamma = 3.
\end{array}
\right. 
\end{eqnarray*}

\begin{Thm}\label{main2}
Let  $\gamma > 1$ and $(u_0, v_0) \in X_{m,\delta}$. We assume that $\mu=1$ and $0 \leq \lambda \leq  2$ or  $\mu>1$ and $ \lambda \geq 0$.
Suppose for some constants $u_{-}>0$ and $v_{-} \in \R$ that
\begin{eqnarray}
\lim_{x\rightarrow - \infty} (u_0 (x), v_0 (x)) = (u_{-} , v_{-})
\end{eqnarray} and
$$
\sup_{(x,y) \in \R^2 }|s(0,x) +r(0,y)| <   \frac{1}{\gamma-1} {u_-}^{-(\gamma-1)/2}.
$$
Then there exist $M_1 >0$ and $\delta_1 >0$ such that $(u(t),v (t)) \in X_{m_1, \delta_1}$ for all $t \in [0,T^*)$.
Furthermore, in addition to the above assumptions, we assume for some constant $K_{m ,\delta} >0$ that
\begin{eqnarray} \label{blocon}
r_x (0,x_0 ) \leq -K_{m ,\delta} \ \mbox{or} \ s_x (0,x_0 ) \leq -K_{m ,\delta}
\end{eqnarray}
with some $x_0 \in \R$.
Then we have $T^* < \infty$.
\end{Thm}
The first half in Theorem \ref{main2} give a sufficient condition for the boundedness of $u$, $v$ and $p'$ on $[0,T^*)$.
The second half give a sufficient condition that the derivative blow-up occurs with $u$, $v$ and $p'$ bounded.
By using the method of  Cheng,  Pan and  Zhu \cite{GPZ}, when  $1 <\gamma <3$, $\mu\geq 1$ and $0 \leq \lambda \leq 2$, we can relax assumptions in Theorem \ref{main2} as follows:
\begin{Thm}\label{main1}
Let $1 <\gamma <3$, $\mu\geq 1$ and $0 \leq \lambda \leq 2$ and $(u_0, v_0) \in X_{m,\delta}$. 
If
\begin{eqnarray} \label{blocon2}
\sqrt{c_0 (x_0)} r_x (0,x_0 ) < -\dfrac{\lambda}{2} \theta_{\gamma} (u_0 (x_0)) 
\end{eqnarray}
or
\begin{eqnarray}
\sqrt{c_0 (x_0)} s_x (0,x_0 ) < -\dfrac{\lambda}{2} \theta_{\gamma} (u_0 (x_0)) 
\end{eqnarray}
is satisfied for some $x_0 \in \R$.
Then we have $T^* < \infty$ 
and there exist $m_1 >0$ and $\delta_1 >0$ such that $(u(t),v (t)) \in X_{m_1, \delta_1}$ for all $t \in [0,T^*)$.
\end{Thm}

 \begin{Remark} \label{prest}{\bf (Remark on the restriction of $\gamma$).}
We can show the similar result to in Theorem \ref{main2} for all $\gamma \in \R \setminus  \{ -1 \}$ (we note that the equations \eqref{de} is linear, when $\gamma=-1$).
In fact, we can show similar  estimates to in Lemmas \ref{esP} \ref{inf0} and \ref{upes} in the same way as their proofs.
However, when $\gamma<-1$, the assumption \eqref{blocon} is replaced by
\begin{eqnarray*} 
r_x (0,x_0 ) \geq  K_{m ,\delta} \ \mbox{or} \ s_x (0,x_0 ) \geq  K_{m ,\delta}.
\end{eqnarray*}
While, it would be difficult to remove the restriction $1 <\gamma <3$ in Theorem \ref{main1} (see Remark \ref{bpp}). 
\end{Remark}

\subsection{Preliminaries for the proofs of Theorems \ref{main2} and \ref{main1}}
First, we introduce some useful identities, which are based on Lax's formulas in \cite{lax}.
For $c=\sqrt{-p' (u)}$, the plus and minus characteristic curves are solutions to the following deferential equations:
\begin{eqnarray*}
\frac{x_{\pm}}{dt}(t) = \pm  c(t,u(t,x_{\pm} (t))).
\end{eqnarray*}
Riemann invariants $r$ and $s$ (see \eqref{ri} for their definitions) are solutions to 
\begin{eqnarray}\label{ww}\left\{
\begin{array}{ll} 
\partial_- r =-\dfrac{\lambda}{2(1+t)^\mu}(r+s),\\
 \partial_+ s =-\dfrac{\lambda}{2(1+t)^\mu}(r+s),
\end{array}\right.
\end{eqnarray}
where $\partial_{\pm} = \partial_t \pm c \partial_x$. We put $A(t)=\exp(\int_0 ^t \frac{\lambda}{2(1+\tau)^\mu} d\tau)$.
 These equations can be written as
\begin{eqnarray}\label{ww2}\left\{
\begin{array}{ll} 
\partial_- ( A(t) r) =- \dfrac{\lambda A(t) s}{2(1+t)^\mu},\\
\partial_+ (A(t) s) =- \dfrac{\lambda A(t) r}{2(1+t)^\mu}.
\end{array}\right.
\end{eqnarray}
While, differentiating the equations in \eqref{ww} with $x$, from the identity $s_x - r_x = 2 \eta_x =-2c u_x$, we have
\begin{eqnarray}\label{rsx}\left\{
\begin{array}{ll} 
\partial_- r_x =\dfrac{c'}{2c}r_x (s_x -r_x) -\dfrac{\lambda}{2(1+t)^\mu}(r+s),\\
 \partial_+  s_x =\dfrac{c'}{2c}r_x (r_x -s_x) -\dfrac{\lambda}{2(1+t)^\mu}(r+s),
\end{array}\right.
\end{eqnarray}
Multiplying  the both side of the equations in \eqref{rsx} by $A(t) \sqrt{c}$, we have
\begin{eqnarray} \label{rs}\left\{
\begin{array}{ll} 
\partial_- y =-A(t)^{-1} \dfrac{\gamma +1}{4}u^{\frac{\gamma -3}{4}}y^2  - \dfrac{\lambda q}{2(1+t)^{\mu}},\\
\displaystyle \partial_+ q =-A(t)^{-1}  \dfrac{\gamma +1}{4}u^{\frac{\gamma -3}{4}}q^2  - \dfrac{\lambda y }{2(1+t)^{\mu}},
\end{array}\right.
\end{eqnarray}
where $y= A(t) \sqrt{c} r_x$ and $q= A(t) \sqrt{c} s_x$. 
Now we rewrite \eqref{rs} as  integral equalities.
Since it holds that 
\begin{eqnarray}
\sqrt{c}s_x (t,x_{-} (t)) =   \frac{d}{dt}\theta_{\gamma}  (u(t,x_{-} (t))), \label{s-1}
\end{eqnarray}
from integration by parts, we  have
\begin{eqnarray}
\int_0 ^t \dfrac{\lambda q(t,x_{-} (\tau))}{2(1+\tau)^{\mu}} d\tau =\int_0 ^t \dfrac{2 A(\tau)}{\lambda(1+\tau)^{\mu}} \left( \dfrac{\mu}{(1+\tau)} - \dfrac{\lambda }{2(1+\tau)^{\mu}} \right) \lambda \theta_{\gamma} (\tau,x_{-} (\tau))d\tau \notag \\
 + \frac{\lambda}{2} \left( \dfrac{A(t) \theta_{\gamma}(t,x_{-}(t)) }{(1+t)^{\mu}}- \theta_{\gamma} (0,x_{-} (0))\right). \label{y-int}
\end{eqnarray}
From the first equation in \eqref{rs} and \eqref{y-int}, $y$  can be written on the minus characteristic curve through $(t,x)$ as follows:
\begin{align}
y(t,x)  =& y(0,x_{-} (0)) -\int_0 ^t \dfrac{\lambda A(\tau)}{2(1+\tau)^{\mu}} \left( \dfrac{\mu}{(1+\tau)} - \dfrac{\lambda }{2(1+\tau)^{\mu}} \right) \lambda \theta_{\gamma} (\tau,x_{-} (\tau))d\tau \notag \\
& - \frac{\lambda}{2} \left( \dfrac{A(t) \theta_{\gamma}(t,x_{-}(t)) }{(1+\tau)^{\mu}}- \theta_{\gamma} (0,x_{-} (0))\right) \notag \\
&-\int_0 ^t A(\tau)^{-1} \dfrac{\gamma +1}{4}u^{\frac{\gamma -3}{4}}y^2 (\tau,x_{-} (\tau)) d\tau. \label{eq1}
\end{align}
In the same way as above, we can obtain the similar identity for $q$.  
Next, we introduce  key inequalities  which ensure the boundedness of  $v$ and $p'$ in Theorem \ref{main2} and \ref{main1}.
We put $\Phi(t,x,y) = r(t,x) +s(t,y)$ and $\Psi(t,x,y)= s(t,x)-r(t,y)$.
\begin{Lemma} \label{esP} 
Let $\gamma >1$, $\mu \geq 0$   and $\lambda \geq 0$. For $C^1$ solution of \eqref{de}, it  holds a priory that
\begin{eqnarray}
\|\Phi (t) \|_{L^{\infty} (\R^2)} \leq \|\Phi (0) \|_{L^{\infty} (\R^2)},\label{pes} \\
\|\Psi (t) \|_{L^{\infty} (\R^2)} \leq \|\Psi (0) \|_{L^{\infty} (\R^2)}.\label{mes}
\end{eqnarray}
\end{Lemma}
\begin{proof}
We only prove \eqref{pes}. \eqref{mes} can be shown in the same way as the proof of \eqref{pes}.
On the characteristic curves $x_{-}(\cdot )$ and $x_+ (\cdot )$ through $(t,x)$ and $(t,y)$ respectively, 
$r$ and $s$ can be written by
\begin{eqnarray}
A(t) r(t,x) =r(0,x_{-} (0))-\int_0 ^t  \dfrac{\lambda A(\tau) s(\tau , x_{-}(\tau)) }{2(1+\tau)^{\mu}} d \tau \label{rri},  \\
A(t) s(t,y) =s(0,x_{+} (0))-\int_0 ^t  \dfrac{\lambda A(\tau) r(\tau , x_{+}(\tau)) }{2(1+\tau)^{\mu}}  d \tau. \label{ssi}
\end{eqnarray}
Summing up the above equations, we have
\begin{align} 
A(t) \Phi(t,x,y) =&r(0,x_{-} (0))+s(0,x_{+} (0))  \notag \\
&- \int_0 ^t  \dfrac{\lambda A(\tau) (s(\tau , x_{-}(\tau))+ r(\tau , x_{+}(\tau)))}{2(1+\tau)^{\mu}} d \tau. \label{risi}
\end{align}
Taking $L^{\infty}$-norm with $(x,y) \in \R^2$ in \eqref{risi}, we obtain that
\begin{eqnarray*}
A(t) \|\Phi (t) \|_{L^{\infty} (\R^2)}\leq \|\Phi (0) \|_{L^{\infty} (\R^2)}+ \int_0 ^t  \dfrac{\lambda  A(\tau) \|\Phi (\tau) \|_{L^{\infty} (\R^2)}}{2(1+\tau)^{\mu}} d \tau.
\end{eqnarray*}
We check that \eqref{pes} can be obtained by the Gronwall inequality.
We put 
$$\Xi (t) = \|\Phi (0) \|_{L^{\infty} (\R^2)}+ \int  \frac{\lambda A(\tau)  \|\Phi (\tau) \|_{L^{\infty} (\R^2)}}{2(1+\tau)^{\mu}} d \tau.$$
 From the above inequality,
it follows that
\begin{eqnarray*}
\dfrac{d \Xi (t)}{dt} \leq \dfrac{\lambda}{2(1+t)^{\mu}} \Xi (t).
\end{eqnarray*}
Multiplying the both side in the above inequality by $A(t)^{-1}$, we have
\begin{eqnarray*}
\dfrac{d}{dt} ( A(t)^{-1} \Xi (t))\leq 0
\end{eqnarray*}
Hence we have $\Xi(t) \leq A(t) \Xi(0)=A(t) \Phi (0) $, which implies \eqref{pes}.
\end{proof}

Lemma \ref{esP} implies that the boundedness of $v$ and $\eta$. Therefore, we have
\begin{eqnarray} \label{lb}
u(t,x) \geq \delta_1 \  \mbox{and} \ \| v (t) \|_{L^\infty} \leq \frac{m_1}{2} 
\end{eqnarray}
for some $m_1 >0$ and $\delta_1 >0$. In order to get a upper bound of $u$ in Theorem \ref{main2}, we prepare the following lemma.
\begin{Lemma} \label{inf0}
Let $\gamma >1$, $\mu \geq 0$ and $\lambda \geq 0$. Suppose that $\lim_{x\rightarrow - \infty}   (u_0 (x), v_0 (x))=(u_-, v_-)$ for $u_- >0$ and $v_- \in \R$.
Then we have that for $C^1$ solutions
\begin{eqnarray} \label{infuu}
\lim_{x\rightarrow - \infty} u (t,x)=u_-
\end{eqnarray}
for all $t \in [0,T^*)$.
\end{Lemma}
\begin{proof}
We put $$ S(t)= \limsup_{x \rightarrow - \infty} 2 \eta (t,x) \ \mbox{and} \   I (t) = \liminf_{x \rightarrow  \infty} 2\eta (t,x) .$$ 
From \eqref{rri} and \eqref{ssi}, we have 
\begin{align*}
2 A(t) \eta (t,x)  =& s(0,x_{+} (0))- r(0,x_{-} (0)) \\
&+  \int_0 ^t  \dfrac{\lambda A(\tau) (s(\tau , x_{-}(\tau))- r(\tau , x_{+}(\tau)))}{2(1+\tau)^{\mu}} d \tau.
\end{align*}
Taking the limit superior,  since $\lim_{x\rightarrow -\infty} x_{\pm} (t) = -\infty$,  we have 
\begin{align*}
A(t)S(t)= & \frac{4}{\gamma-1} {u_-}^{-(\gamma-1)/2} +  \limsup_{x \rightarrow - \infty} \int_0 ^t \dfrac{\lambda A(\tau) (s(\tau , x_{-}(\tau))- r(\tau , x_{+}(\tau)))}{2(1+\tau)^{\mu}} d \tau \\
 =& \frac{4}{\gamma-1} {u_-}^{-(\gamma-1)/2} +  \lim_{R \rightarrow  \infty} \sup_{x\leq -R}   \int_0 ^t \dfrac{\lambda A(\tau) (s(\tau , x_{-}(\tau))- r(\tau , x_{+}(\tau)))}{2(1+\tau)^{\mu}} d \tau \\
\leq&   \frac{4}{\gamma-1} {u_-}^{-(\gamma-1)/2} \\
& +  \lim_{R \rightarrow  \infty}    \int_0 ^t \dfrac{\lambda A(\tau)  \sup \{ r(\tau , x_{+}(\tau))- s(\tau , x_{-}(\tau))\ | \  x\leq -R  \}}{2(1+\tau)^{\mu}} d \tau \\
= &    \frac{4}{\gamma-1} {u_-}^{-(\gamma-1)/2} +     \int_0 ^t \dfrac{\lambda A(\tau) S(\tau)}{2(1+\tau)^{\mu}} d \tau ,
\end{align*}
where the Lebesgue convergence theorem is used in the above forth equality.
Solving the above integral inequality for $S(t)$, we get $S(t) \leq  \frac{4}{\gamma-1} {u_-}^{-(\gamma-1)/2}$.
While, in the same way as the proof of the above estimate for $S(t)$, we have $I(t) \geq \frac{4}{\gamma-1} {u_-}^{-(\gamma-1)/2}$.
Hence we have $S(t) =I(t) =  \frac{4}{\gamma-1} {u_-}^{-(\gamma-1)/2},$ which yields that the convergence of the desired limit and \eqref{infuu}.
\end{proof}

\begin{Remark} \label{plimit}{\bf (Remark on the proof of Lemma \ref{inf0}
).}
In Lemma \ref{inf0}, it would be possible to show the existence of the limit $\lim_{x\rightarrow - \infty} v (t,x)$, if we recall the process of the proofs of the local existence of solutions with $C^1 _b$ initial data
in Friedrichs \cite{KF}, Lax \cite{lax0} and Li and Wu \cite{LW}. However, the existence of this limit is not necessary to get the upper bound of $u$.
The above proof is done without recalling their proofs. 
\end{Remark}

The following Lemma, which can be shown by the inequality \eqref{pes} in Lemma \ref{esP} and Lemma \ref{inf0}, gives a upper bound of $u$.
\begin{Lemma}\label{upes}
 Under the same assumptions as in Lemma \ref{inf0}, we assume that
$$\sup_{(x,y) \in \R^2 }|s(0,x) +r(0,y)| <u_- /2. $$ Then there exists a constant $m_1 >0$ such that
\begin{eqnarray*}
u(t,x) \leq \frac{m_1}{2}
\end{eqnarray*}
for all $(t,x) \in [0,T^*) \times \R$.
\end{Lemma}
\begin{proof}
We chose a constant $\tilde{\delta}>0$ satisfying $$\sup_{(x,y) \in \R^2 }|s(0,x) +r(0,y)|< \tilde{\delta}  < \frac{1}{\gamma-1} {u_-}^{-(\gamma-1)/2}.$$
The first inequality in Lemma \ref{esP}  implies that
\begin{eqnarray*}
2|v(t,x)|= |s(t,x) +r(t,x)| \leq  \| \Phi (t,x,y)\|_{L^{\infty} (\R^2)} \leq  \| \Phi (0,x,y)\|_{L^{\infty} (\R^2)}< \tilde{\delta}.
\end{eqnarray*}
While, from the first inequality in Lemma \eqref{esP}, we have
\begin{eqnarray*}
|\eta (t,x) -\eta (t,y)| - \tilde{\delta} \leq |\eta (t,x) -\eta (t,y)| -|v(t,x)|-|v(t,y)| \leq | \Phi (t,x,y) | <  \tilde{\delta}.
\end{eqnarray*}
Namely we have
\begin{eqnarray*}
|\eta (t,x) -\eta (t,y)|\leq 2 \tilde{\delta}.
\end{eqnarray*}
Taking $y \rightarrow -\infty$ in the above inequality, by Lemma \ref{inf0}, we have
\begin{eqnarray*}
|\eta (t,x) -\frac{2}{\gamma-1} {u_-}^{-(\gamma-1)/2} | \leq 2 \tilde{\delta},
\end{eqnarray*}
 which implies the desired estimate, since $2 \tilde{\delta} <\frac{2}{\gamma-1} {u_-}^{-(\gamma-1)/2}.$
\end{proof}

Next, in order to show Theorem \ref{main1}, we prepare two lemmas whose idea comes from Cheng,  Pan and  Zhu \cite{GPZ}.
\begin{Lemma} \label{tdlb} 
Let $1< \gamma <3 $, $\mu\geq 1$ and $0 \leq \lambda \leq 2$. For $C^1$ solution of \eqref{de},  it holds a priory  that
\begin{eqnarray} \label{tblb1} 
y(t,x) \leq  \tilde{K} \ \mbox{and} \ 
q(t,x)  \leq \tilde{K},
\end{eqnarray}
where the constant $\tilde{K}>0$ depends on $\| u_0 \|_{C_b ^1}$ and $\| v_0 \|_{C_b ^1}$.
\end{Lemma}
\begin{proof}
We note that the second term of the right hand side in \eqref{eq1} 
$$
-\int_0 ^t  \dfrac{A(\tau)}{(1+\tau)^{\mu}} \left( \dfrac{\mu}{(1+\tau)} - \dfrac{\lambda }{2(1+\tau)^{\mu}} \right) \lambda \theta_{\gamma} (\tau,x_{-} (\tau))d\tau
$$
is negative, when  $1< \gamma <3 $, $\mu\geq 1$ and $0 \leq \lambda \leq 2$.
Since the second, the third and the fifth terms of the right hand side in \eqref{eq1} are negative, we have  \eqref{tblb1}.

\end{proof}

\begin{Lemma} \label{upb} 
Let $1< \gamma <3 $ and $\mu\geq 1$ and $0 \leq \lambda \leq 2$. For $C^1$ solution of \eqref{de}, it holds that
\begin{eqnarray}\label{upbes}
u^{\frac{3-\gamma}{4}} (t,x)\leq 
 \left\{ \begin{array}{ll}
 \tilde{K}(1+t)^{1-\frac{\lambda}{2}},\ \ \mbox{if} \ \mu=1 \ \mbox{and}  \ 0 \leq \lambda < 2,\\
\tilde{K} \log (e+t), \ \ \mbox{if} \  \mbox{and} \ \mu=1 \ \mbox{and} \  \lambda = 2,\\
\tilde{K} (1+t), \ \ \mbox{if} \  \mu>1 \ \mbox{and} \  \lambda \geq 0\\
\end{array}
\right. 
\end{eqnarray}
for all $(t,x)$, where $\tilde{K}>0$ is a positive constant depending on  $\| u_0 \|_{C_b ^1}$ and $\| v_0 \|_{C_b ^1}$.
\end{Lemma}
\begin{proof}
From \eqref{ww},  Lemma \ref{tdlb} implies that
\begin{eqnarray*}
r_t  = c r_x  - \dfrac{\lambda(r+s)}{1+t}  \leq  \tilde{K}  A(t)^{-1} \sqrt{c}-  \dfrac{\lambda(r+s)}{2(1+t)^{\mu}}.
\end{eqnarray*}
Similarly we have 
\begin{eqnarray*}
s_t  = -c s_x  - \dfrac{\lambda(r+s)}{1+t}  \geq   - \tilde{K}  A(t)^{-1} \sqrt{c}-  \dfrac{\lambda(r+s)}{2(1+t)^{\mu}}.
\end{eqnarray*}
Hence we have
\begin{eqnarray*}
-2\eta_t= 2 c {u_t} =(r-s)_t \leq  \tilde{K}  A(t)^{-1} \sqrt{c}.
\end{eqnarray*}
Dividing the both side of the above inequality by $\sqrt{c}$ and integrating with $t$ on $[0,T]$, we have
\begin{eqnarray*}
 u^{\frac{3-\gamma}{4}} \leq \tilde{K} \int_0 ^t A(\tau)^{-1} d\tau.
\end{eqnarray*}
Since $A(t)^{-1} =(1+t)^{-\lambda/2}$ when $\mu=1$ and $A(t)^{-1}$ is bounded  when $\mu>1$, we have \eqref{upbes}.
\end{proof}

\subsection{Proof of Theorems \ref{main2}}
The first half in Theorem \ref{main2} can be shown by  \eqref{lb} and Lemma \ref{upes}.
Now we prove the second half in Theorem \ref{main2} in the case that $r_x (0,x_0 ) \leq -K_{m ,\delta}$ only.
The other case can be proved in the same way. 
We take the minus characteristic curve through $(0,x_0).$
By \eqref{lb} and Lemma \ref{upes}, we have the uniform boundedness of $\theta_{\gamma} (u) $.
We note that the second term is $0$ if $\mu=1$ and $\lambda=2$ in \eqref{eq1} and that the term is bounded if $\mu >1$ or $\mu=1$ and $0 \leq \lambda<2$.
Therefore, we have from  \eqref{eq1} that
\begin{eqnarray}
-y(t,x_{-}(t))\geq -y(0,x_{-} (0))-C+ K \int_{0} ^t A(\tau)^{-1} (-y(\tau,x_{-}(\tau))^2 d\tau,  \label{bl-y}
\end{eqnarray}
where $C$ and $K$ are positive constants depending on $m$ and $\delta$.
Since the condition that $r(0,x_0 ) \leq -K_{m ,\delta}$ is satisfied for some large constant $K_{m ,\delta}$, $-y(0,x_{-} (0))-C$
is positive, which leads the finite time blow of $-y(t,x_{-} (t))$. Here we check it. 
We  put $$F(t) =  -y(0,x_{-} (0))-C+ K \int_{0} ^t A(\tau)^{-1}(-y(\tau,x_{-}(\tau))^2 d\tau .$$
From the  differential inequality \eqref{bl-y}, we have
\begin{eqnarray*}
F'(t)\geq K  A(t)^{-1} F^2 (t). 
\end{eqnarray*}
Solving this differential inequality on $F$, we have
\begin{eqnarray*}
F(t )\geq \left( \dfrac{1}{F(0)} - K\int_{0} ^t  A(\tau)^{-1} d\tau \right).
\end{eqnarray*}
Since  $ \lim_{t\rightarrow \infty} \int_{0} ^t A(\tau)^{-1}d\tau =\infty$ for $\mu >1$ and $\lambda \geq 0$ or $\mu= 1$ and  $0 \leq \lambda \leq 2$, from the positivity of $F(0)$, $F(t)$ blows up in finite time.
Hence $T^*$ is finite.

\subsection{Proof of Theorems \ref{main1}}

We put \begin{eqnarray*}
\tilde{A}(t) = 
 \left\{ \begin{array}{ll}
 (1+t),\ \ \mbox{if} \ \mu=1 \ \mbox{and}  \ 0 \leq \lambda < 2 \ \mbox{or} \  \mu>1 \ \mbox{and} \  \lambda >0,\\
 (1+t)\log (e+t), \ \ \mbox{if} \  \mbox{and} \ \mu=1 \ \mbox{and} \  \lambda = 2.
\end{array}
\right. 
\end{eqnarray*}
Since the second and third terms in the right hand side are negative when $1<\gamma <3$, $\mu \geq 1$ and $ 0\leq \lambda \leq 2$ in  \eqref{eq1},
using  Lemma \ref{upb}, we have 
\begin{eqnarray*}
-y(t,x_{-}(t))\geq -y(0,x_0 )-\dfrac{\lambda}{2}\theta_{\gamma} (u_0 (x_0)) + \tilde{K} \int_{0} ^t \tilde{A}(\tau)^{-1}  (-y(\tau,x_{-}(\tau))^2 d\tau,
\end{eqnarray*}
where  $\tilde{K}$ is a positive constant depending on  $\| u_0 \|_{C^1 _b}$ and $\| v_0 \|_{C^1 _b}$.
If \eqref{blocon2} is satisfied, we have by the definition of $y$ that the right hand side of the above differential inequality is positive at $t=0$.
Hence, since $ \lim_{t\rightarrow \infty} \int_{0} ^t \tilde{A}(\tau)^{-1}d\tau =\infty$ under the assumptions on $\lambda$ and $\mu$ in Theorem \ref{main1}, 
$-y(t,x_{-}(t))$ diverge in finite time, if \eqref{blocon2} is satisfied.
From \eqref{lb} and  Lemma \ref{upb}, we have $(u(t), v(t)) \in X_{m_1, \delta_1}$
for all $t \in [0,T^*)$.

\subsection{Upper and lower estimates of the lifespan of solutions of \eqref{de}}

Here we consider the Cauchy problem \eqref{de} with small initial data as
\begin{eqnarray} \label{small}
(u_0, v_0)=(1+\varepsilon \varphi ,\varepsilon \psi ),
\end{eqnarray}
where $(\psi, \varphi ) \in C^1 _b (\R)$  and $\varepsilon >0 $ is a small parameter.
The aim of this subsection is to give sharp upper and lower estimates of 
the lifespan $T^* $ of solutions for small $\varepsilon >0$.
For initial data, we assume that for $(\psi_-, \varphi_- )\in \R^2$
\begin{eqnarray} \label{limpp}
\lim_{x\rightarrow -\infty} (\varphi  (x), \psi (x)) =( \varphi_-,  \psi_-).
\end{eqnarray}
We note that the assumption \eqref{nosin} is satisfied for fixed $\psi \in C^1 _b (\R)$, if $\varepsilon $ is sufficiently small.
We give a upper estimate of the lifespan for $\mu\geq 1$.
\begin{Thm}\label{main3}
Let  $\gamma >1$ and $( \varphi,  \psi) \in C^1 _b (\R)$.  Suppose that \eqref{limpp} is satisfied and that
\begin{eqnarray} \label{KK}
\psi_x (x_0) \leq -K 
\end{eqnarray}
for a positive constant $K=K(\| \varphi  \|_{L^{\infty} (\R)}, \| \psi   \|_{L^{\infty} (\R)})$ and $x_0 \in \R$.
 Then there exists a number $\varepsilon _0 >0$ such that
\begin{eqnarray} \label{ul}
T^* \leq  \left\{ \begin{array}{ll}
C\varepsilon^{-1}  \ \ \mbox{for} \  \mu>1 \ \mbox{and} \ \lambda\geq 0,\\
C\varepsilon^{-\frac{2}{2-\lambda}}  \ \ \mbox{for} \ \mu=1  \ \mbox{and} \  0\leq \lambda < 2,\\
e^{\frac{C}{\varepsilon }}\ \ \mbox{for} \ \mu=1 \ \mbox{and} \ \lambda =2 .  \\
\end{array}
\right. 
\end{eqnarray}
for $0<\varepsilon \leq  \varepsilon _0$
and for some constants $m >0$ and $\delta>0$, it holds that $(u(t), v(t)) \in  X_{m, \delta}$ with all $t \in [0,T^*)$.
\end{Thm}
The following theorem give a lower estimate of the lifespan for  $\mu \geq 0$ and $\lambda \geq 0$.
\begin{Thm}\label{main4}
Let   $\gamma >1$ and $( \varphi,  \psi) \in C^1 _b (\R)$.  Suppose that \eqref{limpp} is satisfied. There exists a number $\varepsilon _0 >0$ such that if $0< \varepsilon \leq \varepsilon_0 $, then
\begin{eqnarray} \label{ll} \label{lbt}
T^* \geq   \left\{ \begin{array}{ll}
C\varepsilon^{-1}  \ \ \mbox{for} \  \mu>1 \ \mbox{and} \ \lambda\geq 0,\\
C\varepsilon^{-\frac{2}{2-\lambda}}  \ \ \mbox{for} \ \mu=1  \ \mbox{and} \  0 \leq \lambda < 2,\\
e^{\frac{C}{\varepsilon }}\ \ \mbox{for} \ \mu=1 \ \mbox{and} \ \lambda =2 ,  \\
\infty \ \ \mbox{for} \ 0\leq \mu< 1 \ \mbox{and} \ \lambda >0  \ \mbox{or} \ \mu=1 \ \mbox{and} \ \lambda >2. 
\end{array}
\right. 
\end{eqnarray}
\end{Thm}

\subsection{Proof of Theorem \ref{main3}}
For simplicity, we just denote  a constant depending on $\| \varphi  \|_{L^{\infty} (\R)}$ and $\| \psi   \|_{L^{\infty} (\R)}$ by $C^*$ or $K^*$.
From \eqref{pes} in Lemma \ref{esP}, we have
\begin{align}
\| \Phi (t) \|_{L^{\infty} (\R^2)} \leq & \| \Phi (0) \|_{L^{\infty} (\R^2)} \notag \\
\leq & 2\| \varepsilon \varphi  (0) \|_{L^{\infty} (\R)} + \sup_{(x,y) \in \R^2 } | \eta (0,x) -\eta (0,y)  |  \leq  C^* \varepsilon \label{smp}
\end{align}
for sufficiently small $\varepsilon >0$.
Since $\Phi (t,x,x)=2v(t,x)$, \eqref{smp} implies that $\|v(t)\|_{L^{\infty} (\R)} \leq C^* \varepsilon $. 
From Lemma \ref{inf0} and \eqref{smp}, in the same as  the proof of Lemma \ref{upes}, we have
$$ \left|\eta (t,x) -\dfrac{2}{\gamma -1} \right| \leq C^* \varepsilon. $$
Hence we have
\begin{eqnarray}
 \left| u (t,x) - 1 \right| \leq C^* \varepsilon. \label{smes}
\end{eqnarray}
Therefore we have $(u,v) \in X_{m,\delta}$ for some $m>0$ and $\delta >0$.
In the same way as the derivation of \eqref{bl-y}, here we derive a differential inequality from \eqref{rs}, which leads the blow-up of solutions.
Using  $\sqrt{c}s_x (t,x_{-} (t)) =   \dfrac{d}{dt}( \theta_{\gamma}  (u(t,x_{-} (t))) - \theta_{\gamma}  (1))$ instead of \eqref{s-1} in \eqref{eq1}, we have from \eqref{smes}
\begin{eqnarray}
-y(t,x_{-}(t))\geq -y(0,x_0 )-K^* \varepsilon + C^* \int_{0} ^t A(\tau)^{-1} (-y(\tau,x_{-}(\tau))^2 d\tau.  \label{bl-smy}
\end{eqnarray}
Similarly we can obtain the following differential inequality for $q$
\begin{eqnarray}
-q(t,x_{-}(t))\geq -q(0,x_0 )-K^* \varepsilon + C^* \int_{0} ^t A(\tau)^{-1} (-q(\tau,x_{+}(\tau))^2 d\tau.  \label{bl-smq}
\end{eqnarray}
From the definitions of $y$ and $q$, 
if the assumption \eqref{KK} in Theorem \ref{main3} is satisfied for a sufficient large positive constant $K>0$, then $y(0,x_0 )-K^* \varepsilon$ or $q(0,x_0 )- K^* \varepsilon$
is positive.
Therefore, from the above differential inequalities, we have that $-y(t,x_{-}(t))$ or $-q(t,x_{+}(t))$ goes to infinity in finite time.
Moreover, solving the differential inequalities \eqref{bl-smy} and \eqref{bl-smq}, we have \eqref{ul}.

\subsection{Proof of Theorem \ref{main4}} 
First we consider the blow-up cases that $\mu >1$ and $\lambda \geq 0$ or $\mu=1$ and $0 \leq \lambda \leq 2$.
From \eqref{smes} and the boundedness of  $\|v(t)\|_{L^{\infty} (\R)}$, local solutions exist as long as $\|y (t)\|_{L^{\infty}}$ and $\|q (t)\|_{L^{\infty}}$ are bounded.
Now we estimate $y$ and $q$. From \eqref{bl-smy} and \eqref{bl-smq}, we have $-y(t,x)\geq -C_1 \varepsilon$ and $-q(t,x) \geq -C_1 \varepsilon$.
In the same way as the derivations of \eqref{bl-smy} and \eqref{bl-smq}, we obtain
\begin{eqnarray*}
-y(t,x_{-}(t))\leq \tilde{K} \varepsilon + C^* \int_{0} ^t A(\tau)^{-1} (-y(\tau,x_{-}(\tau))^2 d\tau
\end{eqnarray*}
and
\begin{eqnarray*}
-q(t,x_{+}(t))\leq \tilde{K} \varepsilon + C^* \int_{0} ^t A(\tau)^{-1} (-q(\tau,x_{+}(\tau))^2 d\tau.
\end{eqnarray*}
By solving these differential inequalities, we have the lower estimate of the lifespan \eqref{lbt} for the  blow-up cases.
Next we consider the global cases that $\mu <1$ and $\lambda > 0$ or $\mu=1$ and $ \lambda > 2$.
From \eqref{eq1}, we have by the definition of $y$
\begin{align}
|r_x (t,x_{-} (t))|\leq & CA^{-1}(t) \varepsilon + \frac{C \varepsilon}{(1+t)^\mu} +  A^{-1}(t)\int_0 ^t \frac{C\varepsilon A(\tau)}{(1+\tau)^{2\mu}} d\tau \notag \\
& + CA^{-1}(t)\int_0 ^t A(\tau) |r_x (\tau,x_{-} (\tau))|^2 d\tau. \label{tui}
\end{align}
The second  term in the right hand side of \eqref{tui} was obtained in the estimate for the third term in \eqref{eq1}. 
To obtain the third term in \eqref{tui}, we used $1/(1+\tau) \leq C/(1+\tau)^{\mu}$ for $\mu \leq 1$ in the estimate of the second term in \eqref{eq1}.
From the straightforward computation, we can check that $(1+t)^\mu A^{-1}(t) \int_0 ^t A(\tau)(1+\tau)^{-2\mu} d\tau$ is bounded by a constant $C_1 >0$ for $\mu <1$ and $\lambda > 0$ or $\mu=1$ and $ \lambda > 2$.
Hence, from the continuity argument, we show that following estimate holds for $t \in [0,T^*)$ if  $\varepsilon >0$ is small enough.
\begin{eqnarray} 
|r_x (t,x_{-} (t))|\leq 2(2C+C_1)\varepsilon (1+t)^{-\mu}. \label{decay}
\end{eqnarray}
Similarly we have the same estimate for $|s_x (t,x_+(t))|$ as in \eqref{decay}. 
Therefore, we have that $T^* =\infty$.

\section{The Euler equation with time and space dependent damping}

In this section, we consider the Cauchy problem \eqref{de0}.
For the coefficient of the damping term, we  assume that $a(t,x) \in C^1 _b ([0,\infty)\times \R)$ satisfies that either 
\begin{eqnarray} \label{yuki}
\lim_{x\rightarrow -\infty} a(t,x) = a(t),
\end{eqnarray}
and
\begin{eqnarray} \label{kiri}
0 \leq a(t,x) \leq A_1 (1+t)^{-\mu} \ \mbox{and} \ 
| a_t (t,x)|+ | a_x (t,x)|  \leq A_2(1+t)^{-\mu}
\end{eqnarray}
 or
\begin{eqnarray} \label{asu}
0 \leq a(t,x) \leq A_3 (1+|x|)^{-\mu} \ \mbox{and} \ 
| a_t (t,x)|+ | a_x (t,x)|  \leq A_4 (1+|x|)^{-\mu},
\end{eqnarray}
where $\mu >1$ and $A_j >0$ with $j=1,2,3,4$  are constants. 

Now we give a main theorem of this section, which give a sufficient condition for the occurrence of the derivative blow-up under the above condition on $a(t,x)$.
\begin{Thm}\label{main5}
Let $(u_0, v_0) \in X_{m,\delta}$, $\gamma >1$ and $\mu >1$. We assume that $a(t,x)$ satisfies either  \eqref{yuki} and \eqref{kiri} or \eqref{asu}.
Suppose that for some $u_- >0 $ and $v_- \in \R$ \begin{eqnarray} \label{asinf0} 
\lim_{x\rightarrow - \infty} (u_0 (x), v_0 (x)) = (u_{-} , v_{-}). 
\end{eqnarray}
Then there exists $\tilde{\varepsilon } >0$ depending $m$, $\delta$, $A_1$ and $A_3$ such that if 
\begin{eqnarray} \label{asx}
\sup_{(x,y) \in \R^2 }|s(0,x) +r(0,y)| <   \tilde{\varepsilon }
\end{eqnarray}
is satisfied, then we have constants $m_1 >0$ and $\delta_1 >0$  depending on 
$m$, $\delta$, $A_1$ and $A_3$ such that $(u(t), v(t)) \in X_{m_1, \delta_1}$ for all $t \in [0,T^*)$.
Furthermore, suppose for some constant $K$ depending on $m$, $\delta$ and $A_j$ with $j=1,2,3,4$ that
\begin{eqnarray}
r_x (0,x_0 ) \leq -K \ \mbox{or} \ s_x (0,x_0 ) \leq -K
\end{eqnarray}
for some $x_0 \in \R$. Then  $T^* < \infty$. 
\end{Thm}
 This theorem is an analogy of Theorem \ref{main2} and  blow-up theorems for $2\times 2$ hyperbolic systems in Lax \cite{lax}, Zabusky \cite{z},  Klainerman and Majda \cite{km} and  Manfrin \cite{mm}.

\subsection{Preliminaries for the proof of Theorem \ref{main5}}
Now we prepare  some identities for Riemann invariants and its derivatives in the same way as in Section $2$.
We denote $A_\pm (t) = \exp\left( \int_0 ^t  a(t,x_{\pm} (\tau))/2 d\tau\right)$. We set $r$ and $s$ and $\partial_{\pm}$ as in \eqref{ri} and \eqref{ww}.
$r$ and $s$ satisfy that
\begin{eqnarray}\label{wwv}\left\{
\begin{array}{ll} 
\partial_- r =-\dfrac{a(t,x) }{2}(r+s),\\
 \partial_+ s =-\dfrac{a(t,x) }{2}(r+s).
\end{array}\right.
\end{eqnarray}
In the same way as the derivations of \eqref{rri} and \eqref{ssi}, we have
\begin{eqnarray}
A_{-} (t) r(t,x) =  r(0,x_{-} (0))-\int_0 ^t  \dfrac{ A_- (\tau) a(t,x_{-} (\tau)) s(\tau , x_{-}(\tau)) }{2} d \tau, \label{riv}  \\
A_{+}(t) s(t,y) = s(0,x_{+} (0))-\int_0 ^t \dfrac{ A_+ (\tau) a(t,x_{+} (\tau)) r(\tau , x_{+}(\tau)) }{2}  d \tau. \label{siv}
\end{eqnarray}
We set $Y(t)= A_- (t) \sqrt{c(u(t,x_{-} (t)))} r_x (t,x_{-} (t))$. In the same way as the derivation of \eqref{eq1}, we have 
\begin{align}
Y(t)  =& Y(0) +\int_0 ^t \dfrac{d}{d\tau} (A_- (\tau) a (t,x_- (\tau)))  \theta_{\gamma} (\tau,x_{-} (\tau)) d\tau \notag \\
&-   \frac{1}{2}(A_- (t) a (t,x_- (t)) \theta_{\gamma}(t,x_{-}(t))  a(0,x_{-} (0)) \theta_{\gamma} (0,x_{-} (0)))  \notag \\
&-\int_0 ^t A_- (\tau) \dfrac{a_x (\tau ,x_{-} ( \tau))}{2} \sqrt{c} (r+s)  d\tau \notag \\
&-\int_0 ^t A_- (\tau)^{-1} \dfrac{\gamma +1}{4}u^{\frac{\gamma -3}{4}} Y(\tau)^2 d\tau. \label{eq1x}
\end{align}
The following Lemma corresponds to Lemma \ref{inf0}.
\begin{Lemma} \label{inf0x} 
Suppose that $a(t,x)$ satisfies either  \eqref{yuki} and \eqref{kiri} or \eqref{asu}. We assume that $\lim_{x\rightarrow - \infty}   (u_0 (x), v_0 (x))=(u_-, v_-)$.
Then we have that for $C^1$ solutions
\begin{eqnarray} \label{infu}
\lim_{x\rightarrow - \infty} u (t,x)=u_-
\end{eqnarray}
for all $t \in [0,T^*)$.
\end{Lemma}
\begin{proof}
Using \eqref{riv} and \eqref{siv}, We can show this Lemma in the almost same way as the proof of Lemma \ref{inf0}. So we omit it.
\end{proof}
As in Section $2$, we put $\Phi (t,x,y) = r(t,x) + s(t,y)$.
If $a(t,x)$ decays time-like, we can show the following estimate for $\Phi  $.
\begin{Lemma} \label{esPx} 
Suppose that $a(t,x)$ satisfies  \eqref{kiri}.
Then, for $C^1$ solutions, there exists a constant $C_a >0$ depending on $A_1$.
\begin{eqnarray} \label{esP1x}
\|\Phi (t)\|_{L^\infty (\R^2)} \leq C_a \|\Phi (0)\|_{L^\infty (\R^2)}.
\end{eqnarray}
\end{Lemma}
\begin{proof}
From \eqref{wwv}, we have 
\begin{eqnarray*}
 r(t,x) =  r(0,x_{-} (0))-\int_0 ^t  \dfrac{  a(t,x_{-} (\tau)) (r(\tau , x_{+}(\tau))+s(\tau , x_{-}(\tau))) }{2} d \tau ,  \\
 s(t,y) = s(0,x_{+} (0))-\int_0 ^t \dfrac{  a(t,x_{+} (\tau)) (r(\tau , x_{+}(\tau))+s(\tau , x_{-}(\tau)))  }{2}  d \tau. 
\end{eqnarray*}
Summing up the above equality and taking $L^{\infty}$-norm on $\R^2$, we have from \eqref{kiri}
\begin{eqnarray*}
\|\Phi (t)\|_{L^\infty (\R^2)} \leq \|\Phi (0)\|_{L^\infty (\R^2)} + \int_0 ^t A_1 (1+\tau)^{-\mu}\|\Phi (\tau)\|_{L^\infty (\R^2)} d\tau 
\end{eqnarray*}
from which, the Gronwall inequality implies \eqref{esP1x}.
\end{proof}

\begin{Lemma}\label{upest}
Let $(u_0, v_0) \in X_{m,\delta}$.
Suppose that $a(t,x)$ satisfies   \eqref{yuki} and \eqref{kiri}. If $\sup_{(x,y) \in \R^2 }|s(0,x) +r(0,y)| < \tilde{\varepsilon}$ for sufficiently small $\tilde{\varepsilon }>0$ depending on $m$, $\delta$ and $A_1$. then there exists a constant $M_1 >0$ such that
\begin{eqnarray*}
u(t,x) \leq M_1
\end{eqnarray*}
for all $(t,x) \in [0,T^*) \times \R$.
\end{Lemma}
\begin{proof}
The proof of this lemma is completely  same as of Lemma \ref{upes}. So we omit it.
\end{proof}

If $a(t,x)$ decays specially, we use the following Lemma. 
\begin{Lemma}\label{upest2}
Let $(u_0, v_0) \in X_{m,\delta}$ and $x_0 \in \R$.
We assume  that \eqref{asu} and  \eqref{asinf0}  are satisfied. Suppose that $\sup_{(x,y) \in \R^2 }|s(0,x) +r(0,y)| < \tilde{\varepsilon  }$ for sufficiently small $\tilde{\varepsilon }>0$ depending on $m$, $\delta$ and $A_3$. Then there exists a constant 
 $\tilde{C_a}>0$ depending  $m$, $\delta$ and $A_3$ such that 
\begin{eqnarray} \label{klein}
\| \Phi(t) \|_{L^{\infty}} \leq \tilde{C_a} \| \Phi (0) \|
\end{eqnarray}
and
\begin{eqnarray} \label{riz}
\frac{u_-}{4}\leq u(t,x) \leq 4 u_{-}
\end{eqnarray}
for all $(t,x) \in [0,T^*) \times \R$, where 
\begin{eqnarray*}
\tilde{C_a} =\exp\left(4 A_3\int_0 ^{\infty}  (1+ c(8 u_{-} ) \tau )^{-\mu} d\tau \right).
\end{eqnarray*}
\end{Lemma}

\begin{proof}
We prove this lemma by using the bootstrap argument.
We set $$T_m = \sup\{ T>0 \ | \  \mbox{\eqref{klein} and \eqref{riz} hold on} \ [0,T] \}.$$
From the fact that $\tilde{C_a} > 1$ and the continuity of $\| \Phi (t) \|_{L^{\infty}} $, \eqref{klein} holds neat $t=0$.
From  \eqref{asinf0} and \eqref{asx}, in the same way as the proof 
of Lemma \ref{upes}, there exists a positive number $\tilde{\varepsilon_0} $ depending on $u_-$ such that
$|u_0 (x) - u_- | \leq u_- / 2$, if $\tilde{\varepsilon }  \leq \tilde{\varepsilon_0} $.
So \eqref{riz} holds near $t=0$.
Hence we have $T_m >0$ for sufficiently small $\tilde{\varepsilon }>0.$
We show $T_m =T^*$ for sufficiently  small $\tilde{\varepsilon }$.
We suppose that $T_m < T^*$.
We note that \eqref{klein} and \eqref{riz}  hold on $[0,T_m]$.
For some small $t_0>0$, it holds that on $[0, T_m + t_0]$
\begin{eqnarray} \label{riz2}
\frac{u_-}{8}\leq u(t,x) \leq 8 u_{-}.
\end{eqnarray}
From the definition of $x_{\pm} (t)$ and \eqref{riz2}, we have for some $x_0$
\begin{eqnarray} \label{yugio}
|x_{\pm} (t)|  \geq c(8 u_{-} ) t - |x_0|
\end{eqnarray}
on $[0, T_m + t_0]$.
From \eqref{wwv} and \eqref{yugio}, we have
\begin{align*}
\|\Phi (t)\|_{L^\infty (\R^2)} \leq & \|\Phi (0)\|_{L^\infty (\R^2)} + \int_0 ^t A_3 \left((1+|x_{+} (\tau)|)^{-\mu} \right. \\
& +(1+|x_{-} (\tau)|)^{-\mu} \left. \right)\|\Phi (\tau)\|_{L^\infty (\R^2)} d\tau, \\
\leq & \|\Phi (0)\|_{L^\infty (\R^2)} + 2A_3 \int_0 ^t  (1+| c(8 u_{-} ) \tau - |x_0||)^{-\mu} \|\Phi (\tau)\|_{L^\infty (\R^2)} d\tau.
\end{align*}
By the Grownwall inequality, we have on $[0, T_m + t_0]$
\begin{align*}
\|\Phi (t)\|_{L^\infty (\R^2)} \leq & \exp\left(2 A_3\int_0 ^t  (1+| c(8 u_{-} ) \tau - |x_0||)^{-\mu}  d\tau \right)\|\Phi (0)\|_{L^\infty (\R^2)} \\
\leq & \exp\left(4 A_3\int_0 ^{\infty}  (1+ c(8 u_{-} ) \tau )^{-\mu}  d\tau \right)\|\Phi (0)\|_{L^\infty (\R^2)} \\
=& \tilde{C_a} \|\Phi (0)\|_{L^\infty (\R^2)}.
\end{align*}
In the same way as the proof of Lemma \ref{upes}, there exists a number $\tilde{\varepsilon }>0$ depending on $C_a$  and $u_-$ such that
\begin{eqnarray}
\frac{u_-}{4}\leq u(t,x) \leq 4 u_{-}
\end{eqnarray}
on $[0,T_m + t_0]$, if $\Phi(0) \leq  \tilde{\varepsilon }$,
which implies that \eqref{riz} holds  $[0,T_m + t_0]$. Therefore, w have a contraction, which yields that $T_m =T^*$
\end{proof}

\subsection{Proof of Theorem \ref{main5}}

From the uniform boundedness of $\|\Phi  (t)\|_{L^{\infty} (\R^2)}$, we have the uniform boundedness of $v(t,x)$.
Hence \eqref{riz} in Lemma \ref{upest} implies that $(u(t),v(t)) \in X_{m_1 ,\delta_1}$ for some $m_1$ and $\delta_1$ with all $t \in [0,T^*)$.
We show the statement of the second half in Theorem \ref{main5}. We only show it in the case that $r_x (x_0) \leq -K$ is assumed for  some $x_0 \in \R$ and a  sufficiently large number $K$.
In the case that the spatial decay  \eqref{asu} is assumed, from \eqref{riz}, we have $|x_{\pm} (t)|  \geq c(4 u_{-} ) t - |x_0|$ for sufficiently small $\tilde{\varepsilon} $.
Hence, by  this estimate and \eqref{asu} or \eqref{kiri}, we have  that the second, third and forth terms in \eqref{eq1x} are uniformly bounded and that   $C^{-1} \leq A_{\pm}(t) \leq C$
for some constant $C>0$.
Hence we have 
\begin{eqnarray*}
Y(t)\leq Y(0) + K^* - C\int_0 ^t  Y(\tau)^2 d\tau,
\end{eqnarray*}
where $C$ and $K^*$ depend on $m$ $\delta$ and $A_j$ for $j=1, 2, 3, 4$. We note that $K^*$ and $C$ can be 
chosen independent of $x_0$.
Therefore $-Y(t)$ blows up in finite time, if $r_x (0) \leq -K$ is assumed for sufficiently large $K >0$.
 
\begin{Remark}{\bf (Expectation for the global existence of solution).}
We expect that solutions of \eqref{de0} exist globally in time, if $a(t,x) \in C^1 _b ([0,\infty)\times \R)$ 
satisfies that either $a(t,x) \geq (1+t)^{-\mu}$ or $a(t,x) \geq (1+|x|)^{-\mu}$ with $0 \leq  \mu <1$ and some suitable conditions. However, our method for the proof of Theorem \ref{main5} is not applicable
to this method, since our method for Lemmas  \ref{esPx} and \ref{upest2} heavily depends on  the boundedness of $\int_0 ^t a(\tau ,x_{\pm} (\tau)) d\tau$.
\end{Remark}




\begin{thebibliography}{99}

\bibitem{al} W. F. Ames, R. J. Lohner,  Group properties of $u_{tt}=[f(u)u_x]_x $,  Int. J. Non-liner Mech. {\bf 16} (1981)  pp. 39-447.
\bibitem{G}  G. Chen, Formation of singularity and smooth wave propagation for the non-isentropic compressible Euler equations, J. Hyperbolic Differ. Equ., {8\bf } (2011) pp. 671-690.
\bibitem{GPZ}  G. Cheng, R. Pan and S. Zhu,  Singularity formation for compressible Euler equation, arXiv:1408.6775.
\bibitem{NC} N. Cristescu,  Dynamic plasticity, North-Holland, Appl. Math. Mech., 1967.
\bibitem{KF} K.O. Friedrichs, Nonlinear hyperbolic differential equations for functions of two independent variables, Amer. J. Math. {\bf 70} (1948) pp. 555-589.
\bibitem{HWY} F. Hou, I. Witt and H.  Yin, On the global existence and blowup of smooth solutions of 3-D compressible Euler equations with time-depending damping, arXiv:1510.04613.
\bibitem{HY} F. Hou and H.  Yin, On the global existence and blowup of smooth solutions to the multi-dimensional compressible Euler equations with time-depending damping, arXiv:1606.08935.
\bibitem{HL} L. Hsiao and T. P. Liu, Convergence to nonlinear diffusion waves for solutions of a system of hyperbolic conservation laws with damping, Comm. Math. Phys., {\bf 143} (1992)  pp. 599-605.
\bibitem{HL2} L. Hsiao and T. P. Liu, Nonlinear diffusive phenomena of nonlinear hyperbolic systems, Chinese Ann. Math. Ser. B {\bf 14} (1993) pp. 65-480. 
\bibitem{HS} L. Hsiao and D. Serre, Global existence of solutions for the system of compressible adiabatic flow through porous media, SIAM J. Math. Anal. {\bf 27} (1996) pp. 70-77.
\bibitem{km} S. Klainerman and A. Majda, Formation of singularities for wave equations including the nonlinear vibrating string, Comm. Pure Appl. Math. {\bf 33} (1980) pp. 241-263. 
\bibitem{lax0} P.D. Lax, Nonlinear hyperbolic equations, Comm. Pure Appl. Math. {\bf 6} (1953) pp. 231-258.
\bibitem{lax} P. D. Lax, Development of Singularities of solutions of nonlinear hyperbolic partial differential equations, J. Math. Phys. {\bf 5} (1964) pp. 611-613.
\bibitem{LW} T.-T. Li and W.-C. Yu, Boundary Value Problems for Quasilinear Hyperbolic Systems, Duke Univ. Press, Durham (1985).
\bibitem{liu}  T.-P. Liu,  Development of singularities in the nonlinear waves for quasi-linear hyperbolic partial differential equations, J. Differential Equations, 32 (1979), pp. 92-111.
\bibitem{mm} R. Manfrin,  A note on the formation of singularities for quasi-linear hyperbolic systems,  SIAM J. Math. Anal. {\bf 32}  (2000) pp. 261-290.
\bibitem{mm2} R. C. MacCamy, V. J. Mizel,   Existence and nonexistence in the large of solutions of quasilinear wave equations,   Arch. Ration. Mech. Anal. {\bf 25}  (1967) pp. 299-320.
\bibitem{MN} P. Marcati and  K. Nishihara, The $L^p$-$L^q$ estimates of solutions to one-dimensional damped wave equations and their application to the compressible flow through porous media,  J. Differential Equations {\bf 191} (2003) pp. 445-469.
\bibitem{NK} K. Nishihara, Asymptotic behavior of solutions of quasilinear hyperbolic equations with linear damping, J. Differential Equations {\bf 137} (1997)  pp. 384-395.
\bibitem{PZ} R. Pan and Y. Zhu, Singularity formation for one dimensional full Euler equations, J. Differential Equations  {\bf 261} (2016) 7132-7144.
\bibitem{XP1} X. Pan, Remarks on 1-D Euler equations with time-decayed damping, arXiv:1510.08115.
\bibitem{XP2} X. Pan, Global existence of solutions to 1-d Euler equations with time-dependent damping, Nonlinear Anal. {\bf 132} (2016)  pp. 327-336.
\bibitem{XP3} X. Pan, Blow up of solutions to 1-d Euler equations with time-dependent damping. J. Math. Anal. Appl. 442 (2016), no. 2, pp. 435-445.
\bibitem{TS} T. Sideris, Formation of singularities in three-dimensional compressible fluids. Comm. Math. Phys. {\bf 101} (1985) pp.  475-485.
\bibitem{s3} Y. Sugiyama, Degeneracy in finite time of 1D quasilinear wave equations, SIAM J. Math. Anal. {\bf 3} (2016) pp. 847-860. 
\bibitem{s4} Y. Sugiyama, Large time behavior of solutions to 1D quasilinear wave equations,  RIMS K\^oky\^uroku Bessatsu {\bf B37} pp. 113-123.
\bibitem{JW1} J. Wirth,  Solution representations for a wave equation with weak dissipation, Math. Meth. Appl. Sci. {\bf 27} (2004) pp. 101-124.
\bibitem{JW2} J. Wirth,  Wave equations with time-dependent dissipation. I. Non-effective dissipation, J. Differential Equations {\bf 222} (2006) pp. 487-514.
\bibitem{JW3} J. Wirth,  Wave equations with time-dependent dissipation. II. Effective dissipation, J. Differential Equations {\bf 232} (2007) pp. 74-103.
\bibitem{WC} D. Wang and G.-Q. Chen, Formation of singularities in compressible Euler-Poisson fluids with heat diffusion and damping relaxation. J. Differential Equations {\bf 144} (1998)  pp. 44-65. 
\bibitem{z} N. J. Zabusky, Exact solution for the vibrations of a nonlinear continuous model string,  J. Math. Phys. {\bf 3} (1962) pp. 1028-1039. 
\bibitem{hz} H. Zheng, Singularity formation for the compressible Euler equations with general pressure law, J. Math. Anal. Appl., {\bf 438} (2016)  pp. 59-72.
\end{thebibliography}


\end{document}